\documentclass[11pt]{article}
\usepackage{geometry}                
\usepackage{graphicx}
\usepackage{enumerate}
\usepackage{setspace}
\usepackage{amssymb}
\usepackage{amsmath, amsthm}
\usepackage{amsmath}
\usepackage{amsthm}
\usepackage{epstopdf}
\usepackage{footnote}
\newtheorem{thm}{Theorem}
\newtheorem{example}{Example}

\newtheorem{definition}{Definition}

\newtheorem*{remark}{Remark}
\newtheorem{lemma}{Lemma}
\newtheorem{corollary}{Corollary}
\begin{document}
\title{Determining Singularities Using Row Sequences of Pad\'{e}-orthogonal Approximants}
\author{N. Bosuwan\footnote{Results of this article constitute a part of this author's Ph.D. dissertation under the mentorship of E.B. Saff at Vanderbilt University.}  , G. L\'opez Lagomasino\footnote{The research of this author was supported by Ministerio de Econom\'ia y Competitividad under grant MTM2012-36372-C03-01.}, and E.B. Saff\footnote{The research of this author was supported, in part, by the U.S. National Science Foundation under grants DMS-0808093 and DMS-1109266.}   }
\maketitle


\section*{Abstract}

Starting from the orthogonal polynomial expansion of a function $F$ corresponding to a finite positive Borel measure with infinite compact support, we study the asymptotic behavior of certain associated rational functions (Pad\'{e}-orthogonal approximants). We obtain both direct and inverse results relating the convergence of the poles of the approximants and the singularities of $F.$ Thereby, we obtain analogues of the theorems of E. Fabry, R. de Montessus de Ballore, V.I. Buslaev, and S.P. Suetin.

 \section*{Keywords}{Pad\'e approximants, orthogonal polynomials, Fabry's theorem, Montessus de Ballore's theorem}

 \section*{2000 MSC.}
Primary 30E10, 41A27; Secondary 41A21.

\section{Introduction}

Let $E$ be an infinite compact subset of the complex plane $\mathbb{C}$ such that $\overline{\mathbb{C}}\setminus E$ is simply connected. There exists a unique exterior conformal representation $\Phi$ from $\overline{\mathbb{C}}\setminus E$ onto $\overline{\mathbb{C}}\setminus \{w: |w|\leq 1\}$ satisfying $\Phi(\infty)=\infty$ and $\Phi'(\infty)>0.$ We assume that $E$ is such that the inverse function $\Psi =\Phi^{-1}$ can be extended continuously to $\overline{\mathbb{C}}\setminus \{w: |w|< 1\}$  (the closure of a bounded Jordan region and a finite interval satisfy the above  conditions). Unless otherwise stated, $E$ will be as described above.

Let $\mu$ be a finite positive Borel measure with infinite support $\mbox{supp}(\mu)$ contained in $E$. We write $\mu \in \mathcal{M}(E)$ and define the associated inner product,
$$\langle g,h \rangle_{\mu}:=\int g(\zeta) \overline{h(\zeta)} d\mu(\zeta), \quad g,h \in L_2(\mu).$$
Let
$$p_{n}(z):=\kappa_n z^n+\cdots, \quad  \kappa_n>0,\quad n=0,1,\ldots,$$ be the orthonormal polynomial of degree $n$ with respect to $\mu$ having positive leading coefficient; that is, $\langle p_n, p_m \rangle_{\mu}=\delta_{n,m}.$  Denote by $\mathcal{H}(E)$ the space of all functions holomorphic in some neighborhood of $E.$

\begin{definition}\textup{
Let $F\in {\mathcal{H}}(E),\, \mu \in \mathcal{M}(E),$ and a pair of nonnegative integers $(n,m)$ be given. A rational function $[n/m]_F^{\mu}:=P_{n,m}^{\mu}/Q_{n,m}^\mu$ is called an $(n,m)$ \emph{Pad\'{e}-orthogonal approximant}  of $F$ with respect to $\mu$ if $P_{n,m}^\mu$ and $Q_{n,m}^\mu$ are polynomials satisfying
\begin{equation}\label{pade2}
\deg(P_{n,m}^\mu) \leq n, \quad \deg(Q_{n,m}^{\mu})\leq m,\quad Q_{n,m}^{\mu}\not\equiv 0,
\end{equation}
\begin{equation}\label{pade3}
\langle Q_{n,m}^{\mu} F-P_{n,m}^{\mu}, p_j \rangle_{\mu}=0, \quad \textup{for $j=0,1,\ldots,n+m.$}
\end{equation}
Since $Q_{n,m}^{\mu} \not\equiv 0$, we normalize it to have leading coefficient equal to $1$.}
\end{definition}

When $E=\{z\in \mathbb{C}: |z|\leq 1\}$ and $d\mu=d\theta/2\pi$ on the boundary of $E$, then $p_n(z) = z^n$ and the Pad\'{e}-orthogonal approximants reduce to the classical Pad\'{e} approximants.  In this case, we write $P_{n,m}, Q_{n,m}$
and $[n/m]_F$, respectively.

The study of the convergence properties of row sequences of Pad\'e approximants (when $m$ is fixed and $n \to \infty$) has a long history beginning with the classical results of J. Hadamard \cite{Had}, R. de Montessus de Ballore \cite{Mon}, and E. Fabry \cite{Fab}. These results have attracted considerable attention motivating different extensions  and generalizations to other approximation schemes using rational functions in which the degree of the denominator remains bounded as $n \to \infty$ (see, for example, \cite{Vib2009}, \cite{CCL}, \cite{Gonchar1}, \cite{Aag81}, \cite{Monte2}, \cite{Monte1}, \cite{Saffpade}, \cite{Sutinpade}, \cite{Sutinpade1}, \cite{sue1}, \cite{sue2}, \cite{Suetin5}). For the case of measures supported on the real line and the unit circle, some results in this direction are contained in \cite{Vib2009}, \cite{CCL0}, \cite{Sutinpade}, and \cite{Sutinpade1}. However, up to the present there are no results of this nature for measures supported on  general compact subsets $E$ of the complex plane. The object of this paper is to fill this gap.

The general theory covers direct and inverse type results. In direct results  one starts with a function for which the analytic properties and location of singularities in a certain domain is known, and using this information one draws conclusions about the asymptotic behavior  of the approximants and their poles. In the inverse direction, the information is given in terms of the asymptotic behavior of the poles of the approximating functions from which the analyticity and location of the singularities of the function must be deduced. We give results in both directions.

For any $\rho>1,$  set
 $$\Gamma_{\rho}:=\{z\in \mathbb{C}: |\Phi(z)|=\rho\}, \quad \quad \mbox{and} \quad \quad \gamma_{\rho}:=\{w\in \mathbb{C}:|w|=\rho\}.$$
Denote by $D_{\rho}$ the bounded connected component of the complement of $\Gamma_{\rho}$ and by $\mathbb{B}(a,\rho)$ the open disk centered at $a \in \mathbb{C}$ of radius $\rho.$ We  call  $\Gamma_\rho$ and $D_{\rho}$ a \emph{level curve} and a \emph{canonical domain} (with respect to $E$), respectively.
We denote by $\rho_0(F)$ the index $\rho(>1)$ of the largest canonical domain $D_{\rho}$ to which $F$ can be extended as a holomorphic function, and by ${\rho_m(F)}$ the index $\rho$ of the largest canonical domain $D_{\rho}$ to which $F$ can be extended as a meromorphic function with at most $m$ poles (counting multiplicities).

Let $\mu \in \mathcal{M}(E)$ be such that
\begin{equation}\label{radius3}
\lim_{n \rightarrow \infty} |p_n(z)|^{1/n}=|\Phi(z)|,
\end{equation}
uniformly inside $\mathbb{C}\setminus E.$ Such measures are called \textit{regular} (cf.  \cite{totik}).  Here and in what follows, the phrase ``uniformly inside a domain'' means ``uniformly on each compact subset of the domain''.  The Fourier coefficient of $F$ with respect to $p_n$ is given by
\begin{equation}\label{Fourierco}
F_n:=\langle F,p_n\rangle_\mu =\int F(z) \overline{p_n(z)}d\mu(z).
\end{equation}
As for Taylor series (see, for example, \cite[Theorem 6.6.1]{totik}),  it is easy to show that
$$\rho_0(F)=\left(\limsup_{n \rightarrow \infty} |F_n|^{1/n} \right)^{-1}.$$
Additionally, the series $\sum_{n=0}^{\infty} F_n p_n(z)$ converges to $F(z)$ uniformly inside  ${D}_{\rho_{0}(F)}$ and diverges pointwise for all $z\in \mathbb{C}\setminus \overline{D_{\rho_0(F)}}.$
Therefore,
if \eqref{radius3} holds, then $$Q_{n,m}^{\mu}(z)F(z)-P_{n,m}^{\mu}(z)=\sum_{k=n+m+1}^{\infty} \langle Q_{n,m}^{\mu}F,p_k\rangle_{\mu}\,p_k(z)$$
for all $z\in D_{\rho_{0}(F)}$ and $P_{n,m}^{\mu}=\sum_{k=0}^{n} \langle Q_{n,m}^{\mu} F,p_k\rangle_{\mu}\,p_k$ is uniquely determined by $Q_{n,m}^{\mu}.$

In contrast with classical Pad\'e approximants, the rational function $[n/m]_{F}^{\mu}$ may not be unique as the following example shows.

 \begin{example}\textup{ Let $E=[-1,1],$ $d\mu=dx/\sqrt{1-x^2}$ and $$F(x)=\frac{37}{x-3}+\sum_{k=0}^{4} c_k p_k(x),$$
 where the $p_k$ are normalized Chebyshev polynomials, and $$c_0:=37, \quad  c_1:=6(-271 \sqrt{\pi}+192 \sqrt{2\pi}), \quad c_2:=-\sqrt{2}+315\sqrt{\pi}-222\sqrt{2\pi},$$
 $$c_3:=3513 \sqrt{\pi}-2484 \sqrt{2\pi}, \quad c_4:= \sqrt{2}+10674 \sqrt{\pi}-7548\sqrt{2\pi}.$$
  Using the program Mathematica it is easy to check that both $Q_{1,2}^{\mu}(x):=x,$ and $Q_{1,2}^{\mu}(x) = (x-3)^2$ satisfy $\langle Q_{1,2}^{\mu}F,p_k\rangle_\mu =0, k=2,3$.  These denominators $Q_{1,2}^\mu$ give us
 \[
 [1/2]_F^{\mu}(x)=\frac{4756\sqrt{\pi}-3363\sqrt{2\pi}-36\sqrt{2\pi}x+144x}{4\sqrt{\pi}x},
 \]
 and
 \[
 [1/2]_F^{\mu}(x)=\frac{1404-28536 \sqrt{\pi}+19827 \sqrt{2\pi}-864x+90364\sqrt{\pi}x-63681 \sqrt{2\pi}x}{ 4\sqrt{\pi}(x-3)^2},
 \]
respectively, which are clearly disinct.}
\end{example}

It is easy to see, however, that the condition
\begin{eqnarray}\label{averbvkaer}
\Delta_{n,m}(F,\mu):=
\begin{vmatrix}
  \langle F, p_{n+1} \rangle_{\mu} &  \langle zF,p_{n+1}\rangle_{\mu} &  \cdots & \langle z^{m-1}F,p_{n+1}\rangle_{\mu} \\
   \vdots   & \vdots    & \vdots &  \vdots  \\
  \langle F, p_{n+m} \rangle_{\mu} & \langle z F,p_{n+m} \rangle_{\mu} &  \cdots &   \langle z^{m-1} F,p_{n+m}\rangle_{\mu} \\
 \end{vmatrix}\not=0
 \end{eqnarray}
and the condition that every solution of (\ref{pade2})-(\ref{pade3}) has $\deg Q_{n,m}^{\mu} = m$ are equivalent. In turn, they imply the uniqueness of $[n/m]_F^{\mu}$.

An outline of the paper is as follows. In Section 2, we state our main results and comment on their connection with classical and recent developments of the theory. Theorem 1 is a direct result whereas Theorems 2-6 are of inverse type. Section 3 is devoted to the proof of Theorem 1. Theorems 2-6 are proved in Section 4.


\section{Main results}
We will make the following assumptions on the asymptotic behavior of the sequence of orthonormal polynomials with respect to a given measure $\mu \in \mathcal{M}(E)$. We write $\mu \in \mathcal{R}(E)$ when the corresponding sequence of orthonormal polynomials has ratio asymptotics; that is,
\begin{equation}\label{ratioasym}
\lim_{n \rightarrow \infty} \frac{p_n(z)}{p_{n+1}(z)}=\frac{1}{\Phi(z)},
\end{equation}
We say that Szeg\H{o} or strong asymptotics takes place, and write $\mu \in \mathcal{S}(E)$, if
\begin{equation}\label{szegoasym}
\lim_{n \rightarrow \infty} \frac{p_n(z)}{c_n \Phi^{n}(z)}=S(z) \quad \textup{and} \quad  \lim_{n\rightarrow \infty} \frac{c_n}{c_{n+1}}=1,
\end{equation}
The first limit in (\ref{szegoasym}) and the one in (\ref{ratioasym}) are assumed to hold uniformly inside  $\overline{\mathbb{C}}\setminus E$, the $c_n$'s are are positive constants, and
$S(z)$ is some holomorphic and non-vanishing function on $\overline{\mathbb{C}}\setminus E.$  Obviously, \eqref{szegoasym} $\Rightarrow$ \eqref{ratioasym} $\Rightarrow$  \eqref{radius3}.

Our first result is the direct type


\begin{thm}\label{montessusana}
 Suppose $F\in \mathcal{H}(E)$ has poles of total multiplicity exactly $m$ in $D_{\rho_{m}(F)}$ at the (not necessarily distinct) points $\lambda_1,\ldots,\lambda_m$ and let $\mu \in \mathcal{R}(E)$. Then, $[n/m]_{F}^{\mu}$ is uniquely determined for all sufficiently large $n$ and the sequence converges uniformly to $F$ inside $D_{\rho_m(F)}\setminus \{\lambda_1,\ldots,\lambda_m\}$ as $n \rightarrow \infty.$ Moreover, for any compact subset $K$ of $D_{\rho_m(F)}\setminus\{\lambda_1,\ldots,\lambda_m\},$
\begin{equation}\label{supnroot}
\limsup_{n\rightarrow \infty} \|F-[n/m]_{F}^{\mu}\|^{1/n}_{K}\leq \frac{\max\{|\Phi(z)|:z\in K\}}{\rho_m(F)},
\end{equation}
where $\|\cdot \|_{K}$ denotes the sup-norm on $K$ and if $K\subset E,$ then $\max\{|\Phi(z)|:z\in K\}$ is replaced by $1.$ Additionally,
\begin{equation} \label{convzeros}
\limsup_{n\rightarrow \infty} \|Q_{n,m}^{\mu} - Q_m\|^{1/n} \leq \frac{\max\{|\Phi(\lambda_j)|:j=1,\ldots,m\}}{\rho_m(F)}<1,
\end{equation}
where $\|\cdot\|$ denotes (for example) the coefficient norm in the space of polynomials of degree $m$ and $Q_m(z) = \prod_{k=1}^{m} (z - \lambda_k)$.
\end{thm}

\begin{remark} When $K=E,$ the rate of convergence in \eqref{supnroot} cannot be improved; that is,
\begin{equation} \label{exact}  \limsup_{n \to \infty} \|F - [n/m]_F^{\mu}\|_E^{1/n} = \limsup_{n \to \infty} \sigma_{n,m}^{1/n} = \frac{1}{\rho_m(F)},
\end{equation}
where
\[ \sigma_{n,m} := \inf_r \|F -r\|_E,
\]
and the infimum is taken over the class of all rational functions of type $(n,m)$
\[ r(z) = \frac{a_nz^n + a_{n-1}z^{n-1} + \cdots + a_0}{b_mz^m + b_{m-1}z^{m-1} + \cdots + b_0}.
\]
We refer the reader to $\cite{Gonchar1, Ed}$ for more information on the second equality in \eqref{exact}.
\end{remark}

In the context of classical Pad\'e approximation, Theorem \ref{montessusana} is known as the Montessus de Ballore theorem (see \cite{Mon}).  In \cite[Theorem 1]{Sutinpade}, S.P. Suetin proves an analogous result for measures supported on a bounded interval of the real line and states without proof that a similar theorem may be obtained for measures supported on a continuum of the complex plane whose sequence of orthonormal polynomials and their associated second type functions have strong asymptotic behavior. The assumptions of our Theorem 1 are substantially weaker.

In the inverse direction we have the following.


\begin{thm}\label{directfabryana}
Let $F\in \mathcal{H}(E)$ and $\mu \in \mathcal{S}(E)$. If
$$\lim_{n \rightarrow \infty} \frac{F_n}{F_{n+1}}=\tau,$$
then $\Psi(\tau)$ is a singularity of $F$ and $\rho_{0}(F)=|\tau|.$
\end{thm}

For expansions in Taylor series and classical Pad\'e approximation this result reduces to Fabry's theorem (see \cite{Fab}).

If $E=\overline{\mathbb{B}}$, where $\mathbb{B}=\mathbb{B}(0,1),$ and the measure $\mu$ supported on $\mathbb{T},$ the unit circle, satisfies the Szeg\H{o} condition
\begin{equation}\label{szegocon}
\int_{0}^{2\pi} \log w(\theta) d\theta>-\infty,
\end{equation}
(where $d\mu(\theta)=w(\theta)d\theta/2\pi+d\mu_{s}(\theta)$ is  the Radon-Nikodym decomposition of $\mu$), it is well known that the orthonormal polynomials $\varphi_n(z) = \kappa_nz^n+ \cdots$ with respect to $\mu$ satisfy the Szeg\H{o} asymptotics (\ref{szegoasym}) (with $c_n=1$). In particular, this allows us to use Theorem \ref{directfabryana} to locate the first singularity  of the reciprocal of the interior Szeg\H{o} function $$S_{\text{int}}(z):=\exp\left(\frac{1}{4\pi}\int_{0}^{2\pi} \log w(\theta)\frac{e^{i\theta}+z}{e^{i \theta}-z} d\theta \right), \quad z\in \mathbb{B},$$ in terms of the  {\it Verblunsky (or Schur) coefficients} $\alpha_n$ ($\alpha_n:=-\overline{\Phi_n(0)}$).  It is well-known that the Szeg\H{o} condition (\ref{szegocon}) also implies that
$$\lim_{n \rightarrow \infty} \kappa_n=\kappa:=\textup{exp}\left\{ - \frac{1}{4 \pi}\int_{0}^{2 \pi} \log w(\theta) d\theta \right\},$$
 and \begin{equation}
\frac{1}{S{_\textup{int}}(z)}=\frac{1}{\kappa} \sum_{k=0}^{\infty} \overline{\varphi_k(0)} \varphi_k(z)\notag
\end{equation}
uniformly inside  $\mathbb{B}$ (see \cite[p. 19-20]{Geronimus}). By Theorem \ref{directfabryana}, we immediately obtain the following.


\begin{corollary}
Let $\mu$ satisfy \textup{(\ref{szegocon})} and assume that $1/S_{\textup{int}}\in \mathcal{H}(\overline{\mathbb{B}})$. Suppose that
$$\lim_{n \rightarrow \infty} \frac{\alpha_n}{\alpha_{n+1}}= \lambda.$$ Then
$\lambda$ is a singularity of $1/S_{\textup{int}}$ and $1/S_{\textup{int}}$ is holomorphic in $\mathbb{B}(0,|\lambda|).$
\end{corollary}

This result complements \cite[Theorem 2]{BLS} where, under stronger assumptions, it is shown that $\lambda$ is a simple pole and $1/S_{\textup{int}}$ has no other singularity in a neighborhood of $\overline{\mathbb{B}(0,|\lambda|)}.$

Using the definition of $Q_{n,1}^{\mu}$ it is easy to verify that whenever $F_{n+1}\not=0$, we have
$$Q_{n,1}^{\mu}(z) = z - \frac{\langle z F,p_{n+1} \rangle_{\mu}}{F_{n+1}}.$$
The next result enables one to locate the first singularity  of $F$ using the zeros of $Q_{n,1}^{\mu}.$


\begin{thm}\label{thefirstrow}
Let $F \in \mathcal{H}(E)$ and $\mu \in \mathcal{S}(E)$. If
$$\lim_{n \rightarrow \infty} \frac{\langle zF, p_{n} \rangle_{\mu}}{F_n}=\lambda,$$
then $\lambda$ is a singularity of $F$ and $\rho_{0}(F)=|\Phi(\lambda)|.$
\end{thm}

The proofs of Theorems \ref{directfabryana} and \ref{thefirstrow} are reduced to Fabry's theorem by  using the following result.


\begin{thm}\label{equivfirstpole}
Let  $F\in \mathcal{H}(E)$ and $\mu \in \mathcal{S}(E)$. Define $f(w):=F(\Psi(w))$ and denote the Laurent series of $f$ about $0$ by $\sum_{k=-\infty}^{\infty}f_{k} w^{k}.$ Then, the following limits are equivalent:
\begin{enumerate}
\item[$(a)$] $\lim_{n\rightarrow \infty} F_n/F_{n+1} =\tau$,
\item[$(b)$] $\lim_{n\rightarrow \infty} \langle zF, p_n \rangle_{\mu}/F_n =\lambda$,
\item[$(c)$] $\lim_{n\rightarrow \infty} f_n/f_{n+1} = \tau$,
\end{enumerate}
where $\tau$ and $\lambda$ are finite and related by the formula $\Phi(\lambda)=\tau.$
\end{thm}

Theorem \ref{thefirstrow} admits the following extension to general row sequences.


\begin{thm}\label{suetinana}
Let $F \in \mathcal{H}(E)$ and $\mu \in \mathcal{S}(E)$. If for all $n$ sufficiently large, $[n/m]_F^{\mu}$ has precisely $m$ finite poles $\lambda_{n,1}, \ldots, \lambda_{n,m},$ and
\begin{equation}
\lim_{n \rightarrow \infty} \lambda_{n,j}= \lambda_j, \quad j=1,2,\ldots,m,\notag
\end{equation}
($\lambda_1,\ldots, \lambda_m$ are not necessarily distinct), then
\begin{enumerate}
\item[$(i)$] $F$ is holomorphic in $D_{\rho_{\min}}$ where $\rho_{\min}:=\min_{1\leq j \leq m}|\Phi(\lambda_j) |)$\textup{;}
\item[$(ii)$] $\rho_{m-1}(F)=\max_{1\leq j \leq m} |\Phi(\lambda_j)|$\textup{;}
\item[$(iii)$] $\lambda_1, \ldots, \lambda_m$ are singularities of $F;$ those lying in $D_{\rho_{m-1}(F)}$ are poles (counting multiplicities), and $F$ has no other poles in $D_{\rho_{m-1}(F)}$.
\end{enumerate}
\end{thm}

For classical Pad\'e approximants, this theorem was proved by S.P. Suetin in \cite{sue2} (see also \cite{sue1}). In \cite[Theorem 1]{Vib2009}, V.I. Buslaev provides an analogue for measures supported on a bounded interval of the real line. Buslaev reduces the proof of his result to Suetin's statement through an extension of Poincar\'{e}'s theorem on difference equations (see Lemmas \ref{Buslaev2}-\ref{Buslaev3} below). We will follow this approach by proving


\begin{thm}\label{thm6}
Let $F \in \mathcal{H}(E)$ and $\mu \in \mathcal{S}(E)$. Define $f(w):=F(\Psi(w))$ and
denote the Laurent series of $f$ about $0$ by $\sum_{k=-\infty}^{\infty}f_{k} w^{k}$ and the regular part of $f$ by $\hat{f}(w):=\sum_{k=0}^{\infty} f_{k}w^{k}.$  For each fixed $m\geq 1,$ the following conditions are equivalent:
\begin{enumerate}
\item[$(a)$] The poles of $[n/m]_{\hat{f}}$ have finite limits $\tau_1,\ldots,\tau_m$, as $n \rightarrow \infty$.
\item[$(b)$] The poles of $[n/m]_F^{\mu}$ have finite limits $\lambda_1,\ldots,\lambda_m$, as $n \rightarrow \infty$.
\end{enumerate}
Under appropriate enumeration of the sub-indices, the values $\lambda_j$ and $\tau_j$, $j=1,\dots,m,$ are related by the formula $\Phi(\lambda_j)=\tau_j$ for all $j=1,\ldots,m.$
\end{thm}


\section{Proof of Theorem \ref{montessusana}}

The second type functions $s_n(z)$  defined by
\begin{equation*}
s_n(z):=\int \frac{\overline{p_n(\zeta)}}{z-\zeta} d\mu(\zeta), \quad z\in \overline{\mathbb{C}}\setminus \mbox{supp}(\mu),
\end{equation*}
play a major role in the proofs that follow.


\begin{lemma}\label{secondtype}
If  $\mu \in \mathcal{R}(E),$ then $$\lim_{n \rightarrow \infty} p_n(z)s_n(z)=\frac{\Phi'(z)}{\Phi(z)},$$
uniformly inside $\overline{\mathbb{C}}\setminus E.$ Consequently, for any compact set $K\subset \mathbb{C}\setminus E,$ there exists $n_0$  such that  $s_n(z)\not=0$ for all $z\in K$ and $n \geq n_0.$
\end{lemma}


\begin{proof}[Proof] From orthogonality, we get
$$p_n(z)s_n(z)=\int \frac{|p_n(\zeta)|^2}{z-\zeta} d\mu(\zeta),\quad z\notin \mbox{supp}(\mu).$$ Since $p_n$ is of norm $1$ in $L_2(\mu)$, the sequence $\left(\int {|p_n(\zeta)|^2}/({z-\zeta}) d\mu(\zeta)\right)_{n \geq 0}$ forms a normal family in $\overline{\mathbb{C}}\setminus E.$ Consequently, the limit stated follows from pointwise convergence in a neighborhood of infinity. Now, for all $z$ sufficiently large, since $\mu \in \mathcal{R}(E)$ from  \cite[Theorem 1.8]{Brian} it follows that\footnote{We note that in \cite[Theorem 1.8]{Brian} it is assumed that $E$ is a compact set bounded by a Jordan curve. However, as pointed out to us by the author, the result remains valid if $E$ verifies the conditions imposed in this paper.}
 $$\lim_{n \rightarrow \infty}\int \frac{|p_n(\zeta)|^2}{z-\zeta} d\mu(\zeta)=\lim_{n \rightarrow \infty}\sum_{k=0}^\infty \frac{1}{z^{k+1}} \int \zeta^{k} |p_n(\zeta)|^2 d\mu(\zeta)=\sum_{k=0}^\infty \frac{1}{z^{k+1}} \frac{1}{2\pi }\int_{\mathbb{T}} \Psi(w)^{k} \frac{dw}{w i}$$
 $$=\frac{1}{2\pi i}\int_{\mathbb{T}} \frac{1}{w(z-\Psi(w))}dw=\frac{1}{2\pi i}\int_{\Psi(\mathbb{T})} \frac{\Phi'(\zeta)}{\Phi(\zeta)(z-\zeta)}d\zeta=\frac{\Phi'(z)}{\Phi(z)}.$$
Since the function on the right-hand side never vanishes in $\mathbb{C} \setminus E,$ the rest of the statements follow at once.
\end{proof}

\begin{proof}[Proof of Theorem \ref{montessusana}]
For $l=0,1,\ldots,$ from \eqref{ratioasym} it follows that
\begin{eqnarray}\label{banana7}
\lim_{n \rightarrow \infty} \frac{p_{n}(z)}{p_{n+l}(z)}=\frac{1}{\Phi(z)^l}, \quad l=0,1,\ldots,
\end{eqnarray}
and by (\ref{banana7}) and Lemma \ref{secondtype}
\begin{align}\label{banana8}
\lim_{n \rightarrow \infty} \frac{s_{n+l}(z)}{s_{n}(z)}&=\lim_{n \rightarrow \infty} \frac{p_n(z)}{p_{n+l}(z)}\frac{p_{n+l}(z)s_{n+l}(z)}{p_n(z)s_n(z)}=\frac{1}{\Phi(z)^{l}} \frac{\Phi'(z)/\Phi(z)}{\Phi'(z)/\Phi(z)}=\frac{1}{\Phi(z)^{l}} ,
\end{align}
uniformly inside $\mathbb{C} \setminus E$. From (\ref{banana7}) and (\ref{banana8}) we obtain that
\begin{eqnarray}\label{banana9}
\lim_{n \rightarrow \infty} |p_n(z)|^{1/n}=|\Phi(z)|, \qquad \mbox{and} \qquad
\lim_{n \rightarrow \infty} |s_n(z)|^{1/n}=\frac{1}{|\Phi(z)|},
\end{eqnarray}
uniformly inside $\mathbb{C} \setminus E.$

By the definition of Pad\'{e}-orthogonal approximant and the first relation in (\ref{banana9}), we have
\begin{equation}\label{banna1}
Q^{\mu}_{n,m}(z)F(z)-P^{\mu}_{n,m}(z)=\sum_{k=n+m+1}^{\infty} a_{k,n} p_k(z), \quad z\in  D_{\rho_0(F)}.
\end{equation}
Using Cauchy's integral formula and Fubini's theorem, we obtain, for $k \geq n+1$
\begin{align}\label{banana202}
&a_{k,n}:=\langle Q_{n,m}^{\mu}F,p_k \rangle_{\mu}=\int \frac{1}{2 \pi i} \int_{\Gamma_{\rho_1}} \frac{Q_{n,m}^{\mu}(t) F(t)}{t-z} dt\, \overline{p_{k}(z)} \,d\mu(z) \notag\\
&=\frac{1}{2 \pi i} \int_{\Gamma_{\rho_1}} Q_{n,m}^{\mu}(t) F(t) \int \frac{\overline{p_k(z)}}{t-z} d\mu(z) dt =  \frac{1}{2 \pi i} \int_{\Gamma_{\rho_1}} Q_{n,m}^{\mu}(t) F(t) s_k(t) dt,
\end{align}
where $1<\rho_1<\rho_0(F).$ Let $\{\alpha_1,\ldots,\alpha_\gamma\}$ be the set of distinct poles of $F$ in $D_{\rho_{m}(F)}$ and $m_k$ the multiplicity  of $\alpha_k$ so that
$$Q(z):=\prod_{j=1}^{m}(z-\lambda_j)=\prod_{k=1}^{\gamma}(z-\alpha_k)^{m_k}, \quad m := \sum_{k=1}^{\gamma}m_k.$$
Multiplying (\ref{banna1}) by $Q$ and expanding $QQ^{\mu}_{n,m}F-QP^{\mu}_{n,m}\in H(D_{\rho_m(F)})$ in terms of the orthonormal system $\{p_{\nu}\}_{\nu=0}^{\infty}$, we obtain that for $z\in D_{\rho_m(F)},$
\begin{equation}\label{banna2}
Q(z)Q^{\mu}_{n,m}(z)F(z)-Q(z)P^{\mu}_{n,m}(z)=\sum_{k=n+m+1}^{\infty} a_{k,n} Q(z)p_k(z)=\sum_{\nu=0}^{\infty} b_{\nu,n} p_\nu(z),
\end{equation}
where
\begin{eqnarray*}
b_{\nu,n}:=\sum_{k=n+m+1}^{\infty} a_{k,n} \langle Qp_k , p_\nu \rangle_{\mu}, \quad \nu=0,1,\dots.
\end{eqnarray*}

First of all, we will estimate $|a_{k,n}|$ in terms of $|\tau_{k,n}|$, where
\begin{equation}\label{banana201}
\tau_{k,n}:=\frac{1}{2 \pi i} \int_{\Gamma_{\rho_2}} Q_{n,m}^{\mu}(t) F(t) s_k(t) dt, \quad \rho_{m-1}(F)<\rho_2<\rho_{m}(F), \quad k=0,1\ldots.
\end{equation}
Since $\rho_2 > \rho_1$ the integral in (\ref{banana201}) allows a better upper bound than the last integral in (\ref{banana202}). For each $k\geq 0$, the function $Q_{n,m}^{\mu}Fs_{k}$ is meromorphic on $\overline{D_{\rho_2}}\setminus D_{\rho_1} =\{z\in \mathbb{C}: \rho_1 \leq |\Phi(z)| \leq \rho_2 \}$ and has poles at $\alpha_1, \ldots, \alpha_\gamma$ with multiplicities at most $m_1,\ldots, m_\gamma,$ respectively. Applying Cauchy's residue theorem we obtain
\begin{align}\label{banana25}
\frac{1}{2 \pi i}\int_{\Gamma_{\rho_2}} Q_{n,m}^{\mu}(t) F(t) s_k(t)dt-\frac{1}{2 \pi i} \int_{\Gamma_{\rho_1}}Q_{n,m}^{\mu}(t) F(t) s_k(t) dt=\sum_{j=1}^{\gamma} \textup{res}(Q_{n,m}^{\mu} F s_k, \alpha_j),
\end{align}
for $k\geq 0.$
The limit formula for the residue gives
\begin{equation}\label{banana70}
\textup{res}(Q_{n,m}^{\mu}F s_k, \alpha_j)=\frac{1}{(m_j-1)!} \lim_{z \rightarrow \alpha_j} ((z-\alpha_j)^{m_{j}} Q_{n,m}^{\mu}(z) F(z) s_k(z))^{(m_j-1)}.
\end{equation}
Since $s_n(z)\not=0$  for all sufficiently large $n$ and $z\in \mathbb{C}\setminus E$ (see Lemma \ref{secondtype}),   Leibniz' formula allows us to write
$$
  ((z-\alpha_j)^{m_{j}} Q_{n,m}^{\mu}(z) F(z)s_k(z))^{(m_j-1)}=\left( (z-\alpha_j)^{m_{j}} Q_{n,m}^{\mu}(z) F(z) s_n(z)\frac{s_k(z)}{s_n(z)} \right)^{(m_j-1)}$$
$$ = \sum_{p=0}^{m_j-1}
{m_j-1
\choose
p}
((z-\alpha_j)^{m_{j}} Q_{n,m}^{\mu}(z)F(z)s_n(z))^{(m_j-1-p)} \left( \frac{s_{k}(z)}{s_n(z)}\right)^{(p)}.
$$
For $j=1,\ldots,\gamma$ and $ p=0,\ldots,m_j-1$, set
\begin{equation}
\beta_n(j,p):=\frac{1}{(m_j-1)!} \left(
\begin{array}{c}
m_j-1\\
p
\end{array}
\right)
\lim_{z \rightarrow \alpha_j} ((z-\alpha_j)^{m_{j}} Q_{n,m}^{\mu}(z) F(z) s_n(z))^{(m_j-1-p)}, \notag
\end{equation}
(notice that the $\beta_n(j,p)$ do not depend on $k$). So, we can rewrite (\ref{banana25}) as
\begin{equation}\label{banana74}
a_{k,n}=\tau_{k,n}-\sum_{j=1}^{\gamma} \left( \sum_{p=0}^{m_j-1} \beta_n(j,p) \left(\frac{s_k}{s_n} \right)^{(p)}(\alpha_j) \right), \quad   n \geq n_0 \quad  \textup{and}\quad  k=0,1,\ldots .
\end{equation}
Since $a_{k,n}=0,$ for $k=n+1,n+2,\ldots,n+m,$ it follows from (\ref{banana74}) that
\begin{equation}\label{banana75}
\sum_{j=1}^{\gamma} \sum_{p=0}^{m_j-1} \beta_n(j,p) \left(\frac{s_k}{s_n} \right)^{(p)}(\alpha_j)=\tau_{k,n}, \quad k=n+1,\ldots, n+m.
\end{equation}
We view \eqref{banana75} as a system of $m$ equations on the $m$ unknowns $\beta_n(j,p).$ If we show that
\begin{eqnarray}\label{matrixL}
\Lambda_n:=
\begin{vmatrix}
  \left(\frac{s_{n+1}}{s_n}\right)(\alpha_j) &  \left(\frac{s_{n+1}}{s_n}\right)'(\alpha_j) &  \cdots & \left(\frac{s_{n+1}}{s_n}\right)^{(m_j-1)}(\alpha_j)  \\
  \left(\frac{s_{n+2}}{s_n}\right)(\alpha_j) &  \left(\frac{s_{n+2}}{s_n}\right)'(\alpha_j) &  \cdots &\left(\frac{s_{n+2}}{s_n}\right)^{(m_j-1)}(\alpha_j)  \\
   \vdots   & \vdots    & \vdots &  \vdots  \\
  \left(\frac{s_{n+m}}{s_n}\right)(\alpha_j) & \left(\frac{s_{n+m}}{s_n}\right)'(\alpha_j) &  \cdots &  \left(\frac{s_{n+m}}{s_n}\right)^{(m_j-1)}(\alpha_j)  \\
 \end{vmatrix}_{j=1,\ldots,\gamma}\not=0
 \end{eqnarray}
 (this expression represents the determinant of order $m$ in
which the indicated group of columns is written out successively for $j = 1,\ldots,\gamma$),
 then we can express $\beta_n(j,p)$ in terms of $(s_k/s_n)^{(p)}(\alpha_j)$ and $\tau_{k,n},$ for $k=n+1,\ldots n+m.$ In fact,
\begin{eqnarray*}
&\lim_{n \rightarrow \infty} \Lambda_n=\Lambda:=
\begin{vmatrix}
  R(\alpha_j) &  R'(\alpha_j) &  \cdots & R^{(m_j-1)}(\alpha_j)  \\
 R^2(\alpha_j) & (R^2)'(\alpha_j) &  \cdots &(R^2)^{(m_j-1)}(\alpha_j)  \\
   \vdots   & \vdots    & \vdots &  \vdots  \\
 R^m(\alpha_j) & (R^m)'(\alpha_j) &  \cdots &  (R^m)^{(m_j-1)}(\alpha_j)  \\
 \end{vmatrix}_{j=1,\ldots,\gamma}\\
 &=\prod_{j=1}^{\gamma}(m_j-1)!! (-\Phi'(\alpha_j))^{m_j(m_j-1)/2} \Phi(\alpha_j)^{-m_j^2} \prod_{1\leq i< j\leq \gamma} \left(\frac{1}{\Phi(\alpha_j)}-\frac{1}{\Phi(\alpha_i)} \right)^{m_i m_j},
 \end{eqnarray*}
 where $R(z)=1/\Phi(z)$ and $n!!=0!1!\cdots n!$
 (use, for example, \cite[Theorem 1]{Sobczyk} for the last equality). Hence, $\Lambda\not=0$ and, for all sufficiently large $n$, $|\Lambda_n|\geq c_1>0$ where the number $c_1$ does not depend on $n$ (from now on, we will denote some constants that do not depend on $n$ by $c_2,c_3,\ldots$ and we will consider only $n$ large enough so that $|\Lambda_n| \geq c_1>0$).


 Applying Cramer's rule to (\ref{banana75}), we have
\begin{equation}\label{betai}
\beta_n(j,p)=\frac{\Lambda_n(j,p)}{\Lambda_n}=\frac{1}{\Lambda_n}\sum_{s=1}^{m} \tau_{n+s,n} C_n(s,q),
\end{equation}
where $\Lambda_n(j,p)$ is the determinant obtained from $\Lambda_n$ by replacing the column with index $q=(\sum_{l=0}^{j-1} m_{l})+p+1$ (where $m_0:=0$) with the column $[\tau_{n+1,n}\,\,\,\ldots \,\,\,\tau_{n+m,n} ]^{T}$ and $C_n(s,q)$ is the determinant of the $(s,q)^{\textup{th}}$ cofactor matrix of $\Lambda_n(j,p).$ Substituting $\beta_n(j,p)$ in (\ref{banana74}) with the expression in (\ref{betai}), we obtain
\begin{equation}\label{banana101}
a_{k,n}=\tau_{k,n}-\frac{1}{\Lambda_n} \sum_{j=1}^{\gamma} \sum_{p=0}^{m_j-1} \sum_{s=1}^{m} \tau_{n+s,n} C_n(s,q) \left(\frac{s_k}{s_n} \right)^{(p)}(\alpha_j), \quad k \geq n+m+1.
\end{equation}
Let $\delta>0$ and $\epsilon > 0$ be sufficiently small so that $\rho_0(F)-2\delta>1$, $$\{z\in \mathbb{C}:|z-\alpha_j|=\epsilon\}\subset \{z\in \mathbb{C}: |\Phi(z)|\geq \rho_0(F)-\delta\},$$ and
\begin{equation}\label{banana102}
\left(\frac{s_{k}}{s_n}\right)^{(p)}(\alpha_j) =\frac{p!}{2 \pi i} \int_{|z-\alpha_j|=\epsilon}\frac{s_{k}(z)}{s_n(z) (z-\alpha_j)^{p+1}}dz, \quad k=0,1\ldots,\quad p=0,\ldots,m_j-1.
\end{equation} Applying (\ref{banana8}) and (\ref{banana102}), we can easily check that
\begin{equation}\label{banana103}
\left|\left(\frac{s_{k}}{s_n}\right)^{(p)}(\alpha_j) \right|\leq c_2, \quad p=0,\ldots,m_j-1,\quad j=1,\ldots,\gamma,\quad k=n+1,\ldots,n+m,
\end{equation}
for  $n \geq n_1,$ and
\begin{equation}\label{banana104}
\left|\left(\frac{s_{k}}{s_n}\right)^{(p)}(\alpha_j) \right|\leq \frac{c_3}{(\rho_0(F)-2\delta)^{k-n}}, \quad p=0,\ldots,m_j-1,\quad j=1,\ldots,\gamma,\quad k\geq n+m+1,
\end{equation}
 for  $n \geq n_2.$
The inequality (\ref{banana103}) implies that
\begin{equation}\label{banana154}
|C_n(s,q)| \leq (m-1)! c_2^{m-1}=c_4, \quad s,q=1,\ldots,m.
\end{equation}
for $n\geq n_3.$
Combining  (\ref{banana101}), (\ref{banana103}), (\ref{banana104}), (\ref{banana154}), and $|\Lambda_n|\geq c_1>0,$ we see that for $n \geq n_4$
\begin{align}\label{banana156}
|a_{k,n}| &\leq |\tau_{k,n}|+\frac{mc_4c_3}{c_1}  \frac{1}{(\rho_0(F)-2\delta)^{k-n}} \sum_{s=1}^{m} |\tau_{n+s,n}| \notag\\
& \leq   |\tau_{k,n}|+ \frac{c_5}{(\rho_0(F)-2\delta)^{k-n}} \sum_{s=1}^{m} |\tau_{n+s,n}|,\quad\quad \quad\quad \quad   k\geq n+m+1.
\end{align}

Now, we estimate $|b_{\nu,n}|$ in terms of $|\tau_{k,n}|.$
By the Cauchy-Schwarz inequality and the orthonormality of $p_\nu$, we have
\begin{equation}\label{banana100}
|\langle Q p_k, p_\nu \rangle_{\mu}|^2 \leq \langle Q p_k, Qp_k \rangle_{\mu} \langle p_\nu,p_\nu  \rangle_{\mu} \leq \max_{z\in E}|Q(z)|^2 = c_6, \quad k,\nu=0,1,\ldots.
\end{equation}
By (\ref{banana156}), (\ref{banana100}), and the fact that $\sum_{k=n+m+1}^{\infty} (\rho_{0}(F)-2\delta)^{n-k}< \infty$, we obtain,   for $n$ sufficiently large and for all $\nu\geq 0,$
\begin{align}\label{banana12345}
|b_{\nu,n}| &\leq \sum_{k=n+m+1}^{\infty}|a_{k,n}| |\langle Q p_{k}, p_{\nu} \rangle|
 \leq \sqrt{c_6} \sum_{k=n+m+1}^{\infty}|a_{k,n}| \notag\\
 &\leq \sqrt{c_6} \left(\sum_{k=n+m+1}^{\infty} |\tau_{k,n}|+c_5 \sum_{k=n+m+1}^{\infty} \frac{1}{(\rho_{0}(F)-2\delta)^{k-n}}  \sum_{s=1}^{m} |\tau_{n+s,n}| \right) \quad\quad \quad\quad \notag\\
 &\leq c_7 \sum_{k=n+1}^{\infty} |\tau_{k,n}|.
\end{align}


Let $K$ be a compact subset of $D_{\rho_{m}(F)}$ and $\sigma>1$ be such that $K \subset\overline{D_\sigma} \subset D_{\rho_m(F)}.$  Choose $\delta>0$ sufficiently small so that
\begin{equation}\label{yuhyggt1}
\rho_2:=\rho_m(F)-\delta>\rho_{m-1}(F), \quad \rho_0(F)-2\delta>1, \quad  \textup{and} \quad
\frac{\sigma+\delta}{\rho_2-\delta}<1.
\end{equation}
We write (\ref{banna2}) in the form
\begin{equation}\label{banana189}
|Q(z)Q^{\mu}_{n,m}(z)F(z)-Q(z)P^{\mu}_{n,m}(z)|\leq \sum_{\nu =0}^{n+m}  |b_{\nu,n}|| p_\nu(z)|+\sum_{\nu =n+m+1}^{\infty}  |b_{\nu,n}|| p_\nu(z)|.
\end{equation}
Define
$$
A_n^1(z):=  \frac{\sum_{\nu =0}^{n+m}  |b_{\nu,n}|| p_\nu(z)|}{|Q(z)Q^{\mu}_{n,m}(z)|}\quad \textup{and} \quad
A_{n}^2(z):=\frac{\sum_{\nu =n+m+1}^{\infty}  |b_{\nu,n}|| p_\nu(z)|}{|Q(z)Q^{\mu}_{n,m}(z)|},
$$
and let $Q_{n,m}^{\mu}(z):=\prod_{j = 1}^{m_n}(z-\lambda_{n,j}).$
Then (\ref{banana189}) implies
 $$ \left|F(z)-\frac{P_{n,m}^{\mu}(z)}{Q_{n,m}^{\mu}(z)} \right|\leq A_n^1(z)+A_n^{2}(z),$$
for all $z\in \hat{D}_{\sigma}:= \overline{D}_{\sigma}  \setminus( \cup_{n=0}^{\infty}\{\lambda_{n,1},\ldots,\lambda_{n,m_n}\}\cup \{\lambda_1,\ldots,\lambda_m\}).$


Let  us bound $A_n^1(z)$ from above. We will first estimate
$|\tau_{k,n}/Q_{n,m}^{\mu}(z)|$
for $z \in \hat{D}_{\sigma}$ and for $k\geq n+1.$
By definition of $\tau_{k,n}$,
\begin{equation}\label{poiu45}
\frac{\tau_{k,n}}{Q_{n,m}^{\mu}(z)}=\frac{1}{2 \pi i} \int_{\Gamma_{\rho_2}} s_{k}(t) F(t)\frac{Q_{n,m}^{\mu}(t)}{Q_{n,m}^{\mu}(z)}dt, \quad k \geq n+1.
\end{equation}
For $n$ sufficiently large,
$$|s_{k}(t)|\leq \frac{c_8}{(\rho_2-\delta)^{k}}, \quad k\geq n+1.$$
Define $$Q_{n,m,\rho_2}^{\mu}(t)=\prod_{\lambda_{n,j}\in D_{\rho_2}}(t-\lambda_{n,j}).$$ It is easy to see that
$$\left|\frac{t-\zeta}{z-\zeta}\right| \leq c_9,$$
for all $t \in \Gamma_{\rho_2}, z\in \hat{D}_{\sigma}$, and $\zeta\in \mathbb{C}\setminus D_{\rho_2}$ (according to (\ref{yuhyggt1}), $\rho_2>\sigma$).
Then,
\begin{align}
\left|\frac{Q_{n,m}^{\mu}(t)}{Q_{n,m}^{\mu}(z)}\right| \leq {c_9}^{m} \left|\frac{Q_{n,m,\rho_2}^{\mu}(t)}{Q_{n,m,\rho_2}^{\mu}(z)} \right| \leq \frac{c_{10}}{|Q_{n,m,\rho_2}^{\mu}(z)|}, \quad z\in \hat{D}_{\sigma}, \quad t \in \Gamma_{\rho_2}.
\end{align}
By (\ref{poiu45}), we obtain
$$\left|\frac{\tau_{k,n}}{Q_{n,m}^{\mu}(z)}\right| \leq \frac{c_{11}}{|Q_{n,m,\rho_2}^{\mu}(z)|(\rho_2-\delta)^{k}}, \quad z\in \hat{D}_{\sigma},\quad k\geq n+1, \quad n\geq n_5,$$ which implies
\begin{equation}\label{banana5675}
\left|\frac{b_{\nu,n}}{Q_{n,m}^{\mu}(z)}\right| \leq \frac{c_{12}}{|Q_{n,m,\rho_2}^{\mu}(z)|(\rho_2-\delta)^{n}}, \quad z\in \hat{D}_{\sigma}, \quad n\geq n_6.
\end{equation}
Applying (\ref{banana9}) and the maximum modulus principle, we have
\begin{equation}\label{banana211}
|p_{\nu}(z)|\leq c_{13} (\sigma+\delta)^\nu,  \quad z\in \overline{D_{\sigma}} ,\quad \nu\geq 0.
\end{equation}
Using (\ref{banana5675}) and (\ref{banana211}), we obtain that
$$
A_n^1(z)
= \frac{1}{|Q(z)|} \sum_{\nu=0}^{n+m} \frac{|b_{\nu,n}| |p_\nu(z)|}{|Q_{n,m}^{\mu}(z)|}\leq \frac{(n+m+1)c_{12}c_{13} (\sigma+\delta)^{n+m}}{|Q(z) Q_{n,m,\rho_2}^{\mu}(z)|(\rho_2-\delta)^n}, \quad z\in \hat{D}_{\sigma}.$$
Choose $\theta>0$ such that $(\sigma+\delta)/(\rho_2-\delta)<\theta<1.$ Then, for $n$ sufficiently large,
\begin{align}\label{banana6789}
A_{n}^1(z) \leq \frac{c_{14} \theta^{n}}{|Q(z)Q_{n,m,\rho_2}^{\mu}(z)|}, \quad z\in \hat{D}_{\sigma}.
\end{align}

Next, we bound $A_n^2(z)$. Since $\deg(Q P_{n,m}^{\mu})\leq n+m,$ by a computation similar to (\ref{banana202}), we obtain
\begin{equation}\label{banana999}
b_{\nu,n}=\langle Q Q_{n,m}^{\mu} F,p_{\nu} \rangle_{\mu}=\frac{1}{2 \pi i}\int_{\Gamma_{\rho_2}} Q(t) Q_{n,m}^{\mu}(t) F(t) s_{\nu}(t) dt, \quad \nu\geq n+m+1.
\end{equation}
As before, from (\ref{banana9}) and (\ref{banana999}), we have
\begin{equation}\label{banana998}
\frac{|b_{\nu,n}|}{|Q(z)Q_{n,m}^{\mu}(z)|} \leq \frac{c_{15}}{|Q(z)Q_{n,m,\rho_2}^{\mu}(z)|(\rho_2-\delta)^{\nu}}, \quad z\in \hat{D}_{\sigma}, \quad \nu \geq n+m+1,
\end{equation}
for $n \geq n_7.$ Using (\ref{banana211}) and (\ref{banana998}), for $n$ sufficiently large, we obtain
\begin{align}\label{banana3456}
A_n^2(z)\leq \frac{c_{16} (\sigma+\delta)^{n}}{|Q(z)Q_{n,m,\rho_2}^{\mu}(z)|(\rho_1-\delta)^{n}}< \frac{c_{17} \theta^{n} }{|Q(z)Q_{n,m,\rho_2}^{\mu}(z)|},\quad z\in \hat{D}_{\sigma}.
\end{align}
Combining (\ref{banana6789}) and (\ref{banana3456}), for $n$ sufficiently large, we have
\begin{equation}\label{pqwer}
\left| F(z)-\frac{P_{n,m}^{\mu}(z)}{Q_{n,m}^{\mu}(z)} \right|\leq \frac{c_{18} \theta^{n}}{|Q(z)Q_{n,m,\rho_2}(z)|},\quad z\in \hat{D}_{\sigma}.
\end{equation}
Let $T_n(z):=Q(z)Q_{n,m,\rho_2}(z).$ Then, $T_n(z)$ is a monic polynomial of degree at most  $2m.$ Let $\varepsilon>0.$ Clearly,
$$e_n:=\left\{z\in \hat{D}_{\sigma}: \left| F(z)-\frac{P_{n,m}^{\mu}(z)}{Q_{n,m}^{\mu}(z)} \right| \geq \varepsilon \right\} \subset \left\{z\in \hat{D}_{\sigma}:\left|Q(z)Q_{n,m,\rho_2}(z) \right|\leq \frac{c_{18}\theta^{n}}{\varepsilon} \right\}=:E_n.$$
The logarithmic capacity is a monotonic set function and satisfies
$$\textup{cap} \,\{z\in \mathbb{C}: |z^n+a_{n-1} z^{n-1}+\ldots+a_0|\leq \rho^n \}=\rho, \quad \rho>0.$$
Hence, we find that for $n$ sufficiently large
$$\textup{cap}\,e_n\leq \textup{cap} \, E_n \leq \left(\frac{1}{\varepsilon} c_{18} \theta^n \right)^{1/\deg{T_n}}\leq  \left( \frac{1}{\varepsilon} c_{18} \theta^n \right)^{1/2m} \leq c_{19} \theta^{n/2m}.$$ This means that $\textup{cap}\{z\in \overline{D}_{\sigma}: \left| F(z)-\frac{P_{n,m}^{\mu}(z)}{Q_{n,m}^{\mu}(z)} \right| \geq \varepsilon \}=\textup{cap}\, e_n \rightarrow 0$, as $n \rightarrow \infty.$ This proves that $[n/m]_F^{\mu}$ converges in capacity to $F$ on each compact subset of $D_{\rho_m(F)},$ as $n \rightarrow \infty.$ On the other hand, the number of poles of $[n/m]_F^{\mu}$ in  $D_{\rho_m(F)}$ does not exceed $m$. Applying \cite[Lemma 1]{Gonchar1} it follows that $[n/m]_F^{\mu}$ converges to $F$ uniformly inside $D_{\rho_m(F)}\setminus\{\lambda_1,\ldots,\lambda_m\},$ as $n \rightarrow \infty$. In addition, we get that each pole of $F$ in $D_{\rho_m(F)}$ attracts as many zeros of $Q_{n,m}^{\mu}$ as its order. Therefore, $\deg Q_{n,m}^{\mu} = m$ for all sufficiently large $n$ which in turn implies that $[n/m]_F^{\mu}$ is uniquely determined for such $n$. We have obtained \eqref{supnroot} and \eqref{convzeros} except for the rate of convergence exhibited in those relations.

To prove (\ref{supnroot}), let $K$ be a compact subset of  $D_{\rho_m(F)}\setminus \{\lambda_1,\ldots \lambda_m\}.$ Take $\sigma$ to be the smallest positive number  $\geq 1$ such that $K \subset \overline{D_{\sigma}} \subset D_{\rho_{m}(F)},$ and choose an arbitrarily small number $\delta>0$ such that
$\rho_2$ satisfies  (\ref{yuhyggt1}).
Note that what we  proved above implies that
$$\max_{z\in D_{\rho_m(F)}}|Q_{n,m}^{\mu}(z)|\leq c_{20}.$$
From (\ref{banana9}), for $n\geq n_8,$
$$|b_{\nu,n}|=\left| \frac{1}{2 \pi i}\int_{\Gamma_{\rho_2}} Q(t) Q_{n,m}^{\mu}(t) F(t) s_{\nu}(t) dt\right| \leq  \frac{c_{21}}{(\rho_2-\delta)^{\nu}} \quad \nu \geq n+m+1,$$
\begin{equation} \label{44}
|\tau_{k,n}|=  \left| \frac{1}{2 \pi i}\int_{\Gamma_{\rho_2}}  Q_{n,m}^{\mu}(t) F(t) s_{k}(t) dt\right| \leq  \frac{c_{22}}{(\rho_2-\delta)^{k}}, \quad k\geq n+1.
\end{equation} Then, by (\ref{banana12345}) and \eqref{44}, for $n \geq n_9,$
$$|b_{\nu,n}| \leq c_7 \sum_{k=n+1}^{\infty} |\tau_{k,n}| \leq  \frac{c_{23}}{(\rho_2-\delta)^{n}} ,\quad  0 \leq \nu \leq n+m.$$
Using (\ref{banana211}), we can prove that for $z\in K$ and for $n \geq n_{10},$
\begin{equation} \label{additional1} |Q(z)Q_{n,m}^{\mu}(z)F(z)-Q(z)P_{n,m}^{\mu}(z)|\leq \sum_{\nu=0}^{\infty} |b_{\nu,n}| |p_{\nu}(z)|\leq c_{24} \left(\left(\frac{\sigma+\delta}{\rho_2-\delta} \right)+\delta \right)^{n}.
\end{equation}
Consequently, for $n \geq n_{10}$ we have
$$\left|F(z)-\frac{P_{n,m}^{\mu}(z)}{Q_{n,m}^{\mu}(z)}\right| \leq \frac{c_{25}}{|Q(z)Q_{n,m}^{\mu}(z)|} \left( \left(\frac{\sigma+\delta}{\rho_2-\delta} \right)+\delta \right)^{n}, \quad z\in K.$$
Since for $n$ sufficiently large, the zeros of $Q_{n,m}^{\mu}(z)$ are distant from $K,$ it follows that
$$\limsup_{n \rightarrow \infty} \|F-[n/m]_F^{\mu}\|_K^{1/n}\leq \left(\frac{\sigma+\delta}{\rho_2-\delta} \right)+\delta.$$
Letting $\delta\rightarrow 0^{+}$ and $\rho_2 \rightarrow \rho_m(F),$ we obtain (\ref{supnroot}).

Finally, we prove (\ref{convzeros}). We first need to show that for $k=1,\ldots,\gamma,$
\begin{equation}\label{additional3}
\limsup_{n \rightarrow \infty} |({Q_{n,m}^{\mu}})^{(j)}(\alpha_k)|^{1/n}\leq \frac{|\Phi(\alpha_k)|}{\rho_m(F)}, \quad j=0,\ldots,m_k-1.
\end{equation}
Let $\varepsilon>0$ be sufficiently small so that $\overline{\mathbb{B}(\alpha_k,\varepsilon)} \subset D_{\rho_{m}(F)}$ for all $k=1,\ldots,\gamma$ and the disks  $\overline{\mathbb{B}(\alpha_k,\varepsilon)},\, k=1, \ldots,\gamma,$ are pairwise disjoint.
As a consequence of (\ref{additional1}), we have
\begin{equation}\label{additional13}
\limsup_{n \rightarrow \infty} \|(z-\alpha_k)^{m_k} F Q_{n,m}^{\mu} -(z-\alpha_k)^{m_k} P_{n,m}^{\mu}\|_{\overline{\mathbb{B}(\alpha_k,\varepsilon)}}^{1/n} \leq \frac{\|\Phi\|_{\overline{\mathbb{B}(\alpha_k,\varepsilon)}}}{\rho_m(F)},
\end{equation} so by Cauchy's integral formula for the derivative, we obtain
\begin{equation}\label{additional4}
\limsup_{n \rightarrow \infty} \|\left[(z-\alpha_k)^{m_k} F Q_{n,m}^{\mu} -(z-\alpha_k)^{m_k} P_{n,m}^{\mu}\right]^{(j)}\|_{\overline{\mathbb{B}(\alpha_k,\varepsilon)}}^{1/n} \leq \frac{\|\Phi\|_{\overline{\mathbb{B}(\alpha_k,\varepsilon)}}}{\rho_m(F)},
\end{equation}
for all $j\geq 0.$ Since $\varepsilon >0$ can be taken arbitrarily small, this implies that
$$\limsup_{n \rightarrow \infty} |L_{k} Q_{n,m}^{\mu}(\alpha_k)|^{1/n} \leq \frac{|\Phi(\alpha_k)|}{\rho_m(F)},$$
where $L_k:=\lim_{z \rightarrow \alpha_k} (z-\alpha_{k})^{m_k}F(z)\not=0$ (because $F$ has a pole of order $m_k$ at $\alpha_k$). Therefore,
$$\limsup_{n \rightarrow \infty} |Q_{n,m}^{\mu}(\alpha_k)|^{1/n} \leq \frac{|\Phi(\alpha_k)|}{\rho_m(F)}.$$

Proceeding by induction, let $r\leq m_k-1$ and assume that
\begin{equation}\label{additional5}
\limsup_{n \rightarrow \infty} |(Q_{n,m}^{\mu})^{(j)}(\alpha_k)|^{1/n}\leq \frac{|\Phi(\alpha_k)|}{\rho_m(F)}, \quad j=0, \ldots, r-1.
\end{equation}
Let us show that the above inequality also holds for $j=r.$ Using (\ref{additional4}), since $r< m_k,$ we obtain
\begin{equation}\label{additinal7}
\limsup_{n \rightarrow \infty} |[(z-\alpha_k)^{m_k}F Q_{n,m}^{\mu}]^{(r)}(\alpha_k)|^{1/n} \leq \frac{|\Phi(\alpha_k)|}{\rho_m(F)}.
\end{equation}
By the Leibniz formula, we have
\begin{equation*}
[(z-\alpha_k)^{m_k} F Q_{n,m}^{\mu}]^{(r)}(\alpha_k)=\sum_{l=0}^{r} {r \choose l}  [(z-\alpha_k)^{m_k} F]^{(l)}(\alpha_k) (Q_{n,m}^{\mu})^{(r-l)}(\alpha_k).
\end{equation*} Therefore, by (\ref{additional5}), (\ref{additinal7}), and the fact that $L_k\not=0,$ it follows that
$$\limsup_{n \rightarrow \infty} |(Q_{n,m}^{\mu})^{(r)}(\alpha_k)|^{1/n} \leq \frac{|\Phi(\alpha_k)|}{\rho_m(F)}$$
which completes the induction and the proof of \eqref{additional3}.

Let $\{q_{k,s}\}_{k=1,\ldots, \gamma,\, s=0,\ldots, m_k-1}$ be a system of polynomials such that $\deg q_{k,s} \leq m-1$ for all $k,s$ and
$$q_{k,s}^{(i)}(\alpha_j)=\delta_{j,k} \delta_{i,s}, \quad 1\leq j \leq \gamma, \quad 0 \leq i \leq m_j-1.$$
It is not difficult to check that $q_{k,s}$ exist (using for example \cite[Theorem 1]{Sobczyk}).
Then,
$$Q_{n,m}^{\mu}(z)=\sum_{k=1}^{\gamma} \sum_{s=0}^{m_{k}-1} (Q_{n,m}^{\mu})^{(s)}(\alpha_k) q_{k,s}(z)+ Q_{m}(z).$$
This formula combined with (\ref{additional3}) imply
$$\limsup_{n \rightarrow \infty} \|Q_{n,m}^{\mu}- Q_{m}\|^{1/n} \leq \frac{\max_{k=1,\ldots, \gamma} |\Phi(\alpha_k)|}{\rho_m(F)}.$$
\end{proof}

\section{Proofs of inverse type results}


We begin stating two lemmas due to V.I. Buslaev (see  \cite[Theorems 5-6]{Vib2009}). These results constitute the basic tools for proving our inverse type results.  We make use of the following notation. Let $f(w) = \sum_{k=-\infty}^{\infty} f_k w^{k}$ be a Laurent series. We denote the regular part of $f(w)$ by $\hat{f}(w):=\sum_{k=0}^{\infty} f_k w^{k}.$ If $\hat{f}(w)$ is holomorphic at $0,$ we denote by $R_m(\hat{f})$ the  radius of the largest disk centered at the origin  to which $\hat{f}(w)$ can be extended as a meromorphic function with at most $m$ poles (counting multiplicities). Define the annulus
$$T_{\delta,m}(f):=\{w\in \mathbb{C}:e^{-\delta} R_0(\hat{f}) \leq |w| \leq e^{\delta} R_{m-1}(\hat{f})\},$$
where $m\in \mathbb{N}$ and $\delta \geq 0$.
We will use $[\cdot]_n$ to denote the coefficient of $w^n$ in the Laurent series expansion around $0$ of the function in the square brackets.
Set $$U:=\overline{\mathbb{C}} \setminus \overline{\mathbb{B}}.$$


\begin{lemma}[Buslaev \cite{Vib2009}]\label{Buslaev2}
Let $m\in \mathbb{N}, \delta>0$, and let $f(w)=\sum_{n=- \infty }^\infty f_n w^n$ be a Laurent series such that
$$0<R_0(\hat{f})\leq R_{m-1}(\hat{f})< \infty, \quad \mbox{and} \quad \varlimsup_{n \rightarrow \infty} |f_{-n}|^{1/n} \leq R_0(\hat{f}).$$
Assume further that
\begin{equation}\label{buslaevlemma2.1}
\lim_{n \rightarrow \infty} [f \alpha_n \eta_{n,j}]_n R_{m-1}^n(\hat{f}) e^{\delta n}=0, \quad j=0,\ldots, m-1,
\end{equation}
where the functions $\alpha_n, \eta_{n,j} \in H(T_{\delta,m}(f))$ have the limits
\begin{equation*}
\alpha(w):=\lim_{n \rightarrow \infty} \alpha_n (w) \not\equiv 0, \quad  \eta_j(w) :=\lim_{n\rightarrow \infty}\eta_{n,j} (w)= \eta^j (w),\quad j=0,\ldots, m-1,
\end{equation*}
uniformly in $T_{\delta,m}(f),$ $\eta(w)$ is a univalent function in $T_{\delta,m} (f),$ and $\alpha(w)$ has at most $m$ zeros in the annulus $T_{0,m}(f)$. Then the function $\alpha(w)$ has precisely $m$ zeros $\tau_1, \ldots ,\tau_m$ in $T_{0,m}(f)$ and $\lim_{n \rightarrow \infty} \tau_{n,j} =\tau_j,$ where the $\tau_{n,j}, j=1,\ldots,m,$ are poles of the classical approximants $[n/m]_{\hat{f}}(w)$. Moreover, for any functions $K_{n,1},\ldots, K_{n,m},L_{n,1},\ldots,  L_{n,m} \in H(T_{\nu,m}(f)), \, \nu>0,$ that converge to $K_1, \ldots, K_m, \, L_1, \ldots, L_m$ uniformly on $T_{\nu,m}(f),$
\begin{equation}\label{buslaevlemma2.3}
\lim_{n \rightarrow \infty} \frac{\det([f K_{n,i} L_{n,j}]_n)_{i,j=1,\ldots, m }}{\det(f_{n-i-j})_{i,j=0,\ldots, m-1}}=\frac{\det(K_r (\tau_s))_{s,r=1,\ldots, m} \det(L_r(\tau_s))_{s,r =1,\ldots ,m}}{W^2 (\tau_1,\ldots, \tau_m)},
\end{equation}
where $W(\tau_1, \ldots, \tau_m)= \det(\tau_s^{r-1})_{s,r=1,\ldots,m} $ is the Vandermonde determinant of the numbers $\tau_1,\ldots, \tau_m$ \textup{(}for multiple zeros the right-hand side of \textup{(\ref{buslaevlemma2.3})} is defined by continuity\textup{)}. In particular, for any $k_1,\ldots, k_m,q_1, \ldots,q_m \in \mathbb{Z},$ the limits
$$\lim_{n \rightarrow \infty} \frac{\det(f_{n-k_i-q_j})_{i,j=1,\ldots,m}}{\det(f_{n-i-j})_{i,j=0,\ldots,m-1}}=\frac{\det(\tau_s^{k_r})_{s,r=1,\ldots,m} \det(\tau_s^{q_r})_{s,r=1,\ldots,m}}{W^2(\tau_1, \dots, \tau_m)}$$
exist.
\end{lemma}

The assumptions $R_{m-1}(\hat{f})< \infty$ and (\ref{buslaevlemma2.1}) in Lemma 2 can be replaced by the following: the functions $\alpha_n(w)$ and $w^{-j} \eta_{n,j}(w)$ are holomorphic in the set $\overline{\mathbb{C}}\setminus \mathbb{B}(0,e^{-\delta} R_0(\hat{f}))$, and
$$[f \alpha_n \eta_{n,j}]_n=0, \quad j=0,\ldots, m-1, \quad n \geq n_0.$$
Hence, we also have

\begin{lemma}[Buslaev \cite{Vib2009}]\label{Buslaev3}
Let $m \in \mathbb{N}$, $\sigma>1,$ and $f(w)= \sum_{n=-\infty}^{\infty} f_n w^n$  be a holomorphic function in the annulus $\{1 < | w|< \sigma\}.$ Assume further that
\begin{equation}\label{buslaevlemma3.1}
[f \alpha_n \eta_{n,j}]_n=0, \quad j=0,\ldots ,m-1, \quad n \geq n_0,
\end{equation} hold,
where $\alpha_n (w)$  and $w^{-j} \eta_{n,j}(w)$ are holomorphic functions in $U$, the limits
\begin{equation*}
\alpha(w):=\lim_{n \rightarrow \infty} \alpha_n (w) \not\equiv 0, \quad  \eta_j(w) :=\lim_{n\rightarrow \infty}\eta_{n,j} (w)= \eta^j (w),\quad j=0,\ldots, m-1,
\end{equation*}
exist uniformly inside $U\setminus\{\infty\}$, the function $\alpha(w)$ has at most $m$ zeros in $U \setminus\{\infty\},$ and $\eta(w)$ is a univalent function in $U$ such that $\eta(\infty)=\infty$. Then, only one of the following assertions takes place:
\begin{enumerate}
\item[$(i)$] $\hat{f}(w)$ is a rational function with at most $m-1$ poles;
\item[$(ii)$] $\alpha(w)$ has precisely $m$ zeros $\tau_1,\ldots, \tau_m$ in $U \setminus\{\infty\}$, these zeros are singularities of $f(w)$, with an appropriate ordering $|\tau_1|=R_0(\hat{f}),\ldots, |\tau_m|=R_{m-1}(\hat{f}),$ and the limits $\lim_{n \rightarrow \infty} \tau_{n,j} =\tau_j$ exist, where the $\tau_{n,j},j=1,\ldots,m,$ are the poles of the classical Pad\'e approximants $[n/m]_{\hat{f}}(w).$
 \end{enumerate}
\end{lemma}

Define
\[ h_n(w):=c_n w^{n+1} s_n(\Psi(w)) \Psi'(w).
\]

\begin{lemma}\label{functionh}
Let $F\in \mathcal{H}(E).$ Define $f(w):=F(\Psi(w)).$
The functions $h_n(w)$ are holomorphic in $U,$ $F_n=[f h_n]_n/c_n$ and $\langle z F, p_n\rangle_{\mu}= [\Psi f h_n]_n/c_n.$
If $\mu \in \mathcal{S}(E),$  then the sequence $h_n(w)$ converges to some non-vanishing function $h(w)$ uniformly inside $U.$
\end{lemma}

\begin{proof}[Proof]
 Clearly, $h_n(w)$ is holomorphic in $U.$
Let $\epsilon>0$ be a small number so that  $\Gamma_{1+\varepsilon}$ is in the domain of holomorphy of $F(z).$
By Fubini's theorem and Cauchy's integral formula, we have
$$
F_n = \int F(z) \overline{p_n(z)} d\mu(z)
        =  \int  \left( \frac{1}{2 \pi i} \int_{\Gamma_{1+\varepsilon}} \frac{F(\zeta)}{\zeta -z}d\zeta \right) \overline{p_n(z)} d\mu(z)$$$$
        =   \frac{1}{2 \pi i} \int_{\Gamma_{1+\varepsilon}} F(\zeta)        \int \frac{\overline{p_n(z)}}{\zeta -z} d\mu(z)    d\zeta
        = \frac{1}{2 \pi i} \int_{\Gamma_{1+\varepsilon}} F(\zeta) s_n(\zeta) d\zeta $$$$
        = \frac{1}{2 \pi i}  \int_{\gamma_{1+\varepsilon}} f(w) s_n(\Psi(w)) \Psi'(w) dw
        =   \frac{1}{c_n}\frac{1}{2 \pi  i} \int_{\gamma_{1+\varepsilon}} \frac{f(w) h_n(w)}{w^{n+1}} dw
        =  \frac{1}{c_n} [f h_n]_n.
$$
The other formula is obtained similarly.

If $\mu \in \mathcal{S}(E)$ then $\mu \in \mathcal{R}(E)$ and using Lemma \ref{secondtype}, we have
$$
 h(w) := \lim_{n \rightarrow \infty} h_n(w)=\lim_{n \rightarrow \infty}c_n w^{n+1} s_n(\Psi(w)) \Psi'(w)$$$$=w \Psi'(w) \lim_{n \rightarrow \infty} \frac{c_n w^{n}}{p_n(\Psi(w))} \lim_{n \rightarrow \infty}p_n(\Psi(w))s_n(\Psi(w))=\frac{1}{S(\Psi(w))},
$$
uniformly inside $U.$
\end{proof}

\begin{proof}[Proof of Theorem \ref{equivfirstpole}]
Using Lemma \ref{Buslaev3} for $m=1$, we will prove that $(a)$ or $(b)$ imply $(c)$.
Let $$\tau_n:= \frac{F_n}{F_{n+1}}, \quad  \lambda_n: =\frac{\langle zF,p_n \rangle_{\mu}}{F_n}.$$
Set $\eta_{n,0}(w)\equiv 1,  w\in U,$ and define $$\alpha_{n,1}(w):=\frac{c_n}{c_{n+1}}\frac{\tau_{n}  h_{n+1}(w)}{w} -h_n(w), \quad \alpha_{n,2}(w):=\frac{h_{n+1}(w) (\lambda_{n+1}-\Psi(w))}{w}, \quad w\in U.$$ The functions $\alpha_{n,1}(w)$ and $\alpha_{n,2}(w)$ are holomorphic in $U.$
By Lemma \ref{functionh}, for $\varepsilon>0$ sufficiently small so that $f(w)$ is holomorphic in a neighborhood of $\gamma_{1+\varepsilon},$
$$
[f\alpha_{n,1}]_n =\frac{c_n}{c_{n+1}}\frac{\tau_n}{2 \pi i}\int_{\gamma_{1+\varepsilon}} \frac{f(w) h_{n+1}(w)}{w^{n+2}} dw-\frac{1}{ 2 \pi i}\int_{\gamma_{1+\varepsilon}} \frac{f(w) h_n(w)}{w^{n+1}} dw $$$$
                            =\frac{c_n}{c_{n+1}}\frac{F_n}{F_{n+1}} [f h_{n+1}]_{n+1}-[f h_{n}]_{n}
                             =0
$$
and
$$[f \alpha_{n,2}]_n =\frac{\lambda_{n+1}}{2 \pi i} \int_{\gamma_{1+\varepsilon}} \frac{f(w) h_{n+1}(w)}{w^{n+2}} dw-\frac{1}{ 2 \pi i}\int_{\gamma_{1+\varepsilon}}\frac{\Psi(w)f(w) h_{n+1}(w)}{w^{n+2}}dw$$
$$=\frac{c_{n+1}}{c_{n+1}}\frac{\langle zF,p_{n+1} \rangle_{\mu}}{F_{n+1}} [f h_{n+1}]_{n+1}-[\Psi f h_{n+1}]_{n+1}=0.$$
If $(a)$ holds, then
$$\alpha_1(w):=\lim_{n \rightarrow \infty} \alpha_{n,1} (w)= h(w) \left(\frac{\tau}{w}  -1 \right), \quad \textup{uniformly inside $U$},$$ and if $(b)$ holds, then
$$\quad  \alpha_2(w):=\lim_{n \rightarrow \infty} \alpha_{n,2}(w)=\frac{ h(w)(\lambda-\Psi(w))}{w}, \quad \textup{uniformly inside $U$}.$$
Since $h(w)$ is never zero on $U$, each function $\alpha_j(w)$, $j=1,2,$ has at most one zero in $U$ (which is $\tau$). By Lemma \ref{functionh}, $\alpha_1(\infty)=-h(\infty)\not=0$ and $\alpha_2(\infty)=-\textup{cap}(E)h(\infty)\not=0.$
Moreover, if $f_n=0$ for $n \geq n_0,$ then
$F_n=[f h_n]_n=0$ (recall that $h_n(w)$ is analytic at $\infty$). Therefore, by (ii) in Lemma \ref{Buslaev3}, $|\tau|>1$ and $\lim_{n\rightarrow \infty} f_n/f_{n+1}=\tau.$

Now, using Lemma \ref{Buslaev2} for $m=1$, we prove that $(c)$ implies $(a)$ and $(b)$.
Assume that
$\lim _{n\rightarrow \infty} f_n/f_{n+1}=\tau.$ Set $\eta_{n,0}(w)\equiv 1,$
$$\tau_n:=\frac{f_n}{f_{n+1}},\quad \textup{and} \quad \alpha_n(w):=\frac{\tau_n }{w}-1, \quad z\in U.$$ Therefore, $$[f \alpha_n]_n=\tau_n f_{n+1}- f_n=0,$$
$$\alpha(w):=\lim_{n \rightarrow \infty} \alpha_n(w)=\frac{\tau}{w}-1,\quad \textup{uniformly inside $U$},$$
$\alpha(\infty)=-1$, and $\alpha(w)$ has at most one zero in $U.$
Applying (\ref{buslaevlemma2.3}) in Lemma \ref{Buslaev2}, if we select $K_{n,1}(w)=h_n(w)$ and $L_{n,1}(w)=1$, we have
\begin{align*}
\lim_{n \rightarrow \infty} \frac{[f h_n]_n}{f_n}=h(\tau),
\end{align*}
and if we select $K_{n,1}(w)=\Psi(w)h_n(w)$ and $L_{n,1}(w)=1$, we have
$$\lim_{n \rightarrow \infty}\frac{[\Psi f  h_n]_{n}}{f_n}=  \Psi(\tau)h(\tau).$$
Since $h(w)$ vanishes nowhere in the domain $U$,
$$\lim_{n\rightarrow \infty}\frac{F_n}{F_{n+1}}=\lim_{n\rightarrow \infty}\frac{c_{n+1}}{c_{n}}\frac{[fh_n]_n}{[fh_{n+1}]_{n+1}} =\lim_{n \rightarrow \infty}\frac{c_{n+1}}{c_{n}}\frac{[fh_n]_n}{f_n} \frac{f_n}{f_{n+1}}\frac{f_{n+1}}{[fh_{n+1}]_{n+1}}=\tau,$$ and
$$\lim_{n\rightarrow \infty}\frac{ \langle z F, p_n \rangle_{\mu}}{F_n}=\lim_{n \rightarrow \infty} \frac{c_n}{c_{n}}\frac{[\Psi f h_n]_n}{[f h_n]_n}=\lim_{n \rightarrow \infty}\frac{[\Psi f h_n]_n}{f_n} \frac{f_n}{[f h_n]_n}=\Psi(\tau)=\lambda.$$ The proof is complete.
\end{proof}

\begin{proof}[Proof of Theorem \ref{thm6}]
First of all, we prove that $(b)$ implies $(a)$ using Lemma \ref{Buslaev3}. We assume that the zeros of $Q_{n,m}^{\mu}(z)$ have limits $\lambda_1,\ldots,\lambda_m$, as $n \rightarrow \infty$. For $w\in U,$ we define
$$\alpha_n(w):=w^{-m}h(w) Q_{n,m}^{\mu}(\Psi(w)),$$
$$\eta_{n,j}(w):=  \frac{c_{n+m-j}    w^{n+m+1} s_{n+m-j}(\Psi(w)) \Psi'(w)}{h(w)}, \quad j=0,\ldots, m-1.$$
The functions $\alpha_n(w)$ and $w^{-j}\eta_{n,j}(w)=h_{n+m-j}(w)/h(w),$ $j=1,\ldots,m-1,$ are holomorphic in $U$, and
$$ \alpha(w):=\lim_{n \rightarrow \infty} \alpha_n(w)=w^{-m} h(w) \prod_{j=1}^m (\Psi(w)-\lambda_j),$$
$$\eta_j(w):=\lim_{n\rightarrow \infty} \eta_{n,j} (w)=w^{j}, \quad j=0,1, \ldots, m-1,$$ uniformly inside $U\setminus \{\infty\}.$ Since $h(w)$ is never zero in $U$, $\alpha(w)$ has at most $m$ zeros in $U\setminus\{\infty\}.$
 By Cauchy's integral formula, Fubini's theorem, and the definition of $Q_{n,m}^{\mu}$, we have, for $\epsilon>0$ sufficiently small so that $F(z)$ is analytic on $D_{1+\varepsilon}$, and for  $j=0,\ldots, m-1,$
$$[f \alpha_n \eta_{n,j}]_n=\frac{c_n}{2 \pi i} \int_{\gamma_{1+\varepsilon}} F(\Psi(w))Q_{n,m}^{\mu}(\Psi(w)) s_{n+m-j}(\Psi(w))\Psi'(w)dw$$
$$=\frac{c_n}{2 \pi i} \int_{\Gamma_{1+\varepsilon}} F(t) Q_{n,m}^{\mu}(t) s_{n+m-j}(t) dt= \frac{c_n}{2 \pi i}\int_{\Gamma_{1+\varepsilon}} F(t) Q_{n,m}^{\mu}(t) \int \frac{\overline{p_{n+m-j}(z)}}{t-z} d\mu(z) dt$$
$$= c_n\int \frac{1}{ 2\pi i} \int_{\Gamma_{1+\varepsilon}}  \frac{F(t) Q_{n,m}^{\mu}(t)}{t-z}    dt \overline{p_{n+m-j}(z)} d\mu(z)=c_n\int F(z)Q_{n,m}^{\mu}(z) \overline{p_{n+m-j}(z)}d\mu(z) =0.$$
Therefore, the assumptions of Lemma \ref{Buslaev3} are satisfied.
If the regular part of $f(w)$ is a  rational function with at most $m-1$ poles, then $F(z)$ is a rational function with at most $m-1$ poles which implies that
 $\Delta_{n,m}(F,\mu)=0$ for $n$ sufficiently large. This is impossible, because $\deg(Q_{n,m}^{\mu})=m,$ for $n$ sufficiently large.
Therefore,  by Lemma \ref{Buslaev3}, $\alpha(w)$ has precisely $m$ zeros $\tau_1,\ldots, \tau_m$ in $U\setminus \{\infty\}$ and the limits of the poles of the classical Pad\'{e} approximants $[n/m]_{\hat{f}}(w)$ are $\tau_1,\ldots, \tau_m,$ as $n \rightarrow \infty$.

Now, we prove that $(a)$ implies $(b)$ using Lemma \ref{Buslaev2}. Assume that the poles of $[n/m]_{\hat{f}}(w)$ have limits $\tau_1,\ldots,\tau_m$, as $n \rightarrow \infty$.
 We assume further that $Q_{n,m}(w)$ is monic.

Define, for $w\in U,$
$$\tilde{\alpha}_n(w):= w^{-m} Q_{n,m}(w),$$
$$\tilde{\eta}_{n,\nu}(w):= w^{\nu}, \quad \nu=0,\dots, m-1.$$
Then,
$$\tilde{\alpha}(w):=\lim_{n \rightarrow \infty} \tilde{\alpha}_n(z)=w^{-m} \prod_{j=1}^{m} (w-\tau_j),$$
$$\tilde{\eta}_{\nu}(w)= w^{\nu}, \quad \nu=0,\dots, m-1,$$
uniformly inside $U\setminus\{\infty\}.$ By the definition of $Q_{n,m}(z)$, it follows that, for $\epsilon>0$ sufficiently small so that $f(w)$ is holomorphic on $\gamma_{1+\varepsilon}$ and for $n$ sufficiently large,
\begin{align*}
[f \tilde{\alpha}_n \tilde{\eta}_{n,\nu}]_n=[\hat{f} \tilde{\alpha}_n \tilde{\eta}_{n,\nu}]_n= \frac{1}{2 \pi i}\int_{\gamma_{1+\varepsilon}} \frac{ \hat{f}(w) Q_{n,m}(w)}{w^{m-\nu+n+1}} dw=0,\quad  \nu=0,\ldots, m-1.
\end{align*}
We can easily check the rest of the conditions required in Lemma \ref{Buslaev2} for $\tilde{\alpha}_n(w)$ and $\tilde{\eta}_{n,\nu}(w),$ so we can apply the equality (\ref{buslaevlemma2.3}) in Lemma \ref{Buslaev2}.

Next, set
 \begin{equation}\label{qnmcircle}
 \tilde{Q}_{n,m}(z):=
 \begin{vmatrix}
 c_{n+1} \langle F,p_{n+1} \rangle_{\mu} & c_{n+1} \langle zF,p_{n+1} \rangle_{\mu} & \cdots & c_{n+1}  \langle z^m F,p_{n+1} \rangle_{\mu}\\

   \vdots   & \vdots   & \cdots    & \vdots   \\

 c_{n+m} \langle F,p_{n+m} \rangle_{\mu} & c_{n+m} \langle zF,p_{n+m} \rangle_{\mu} &  \cdots & c_{n+m} \langle z^{m} F,p_{n+m} \rangle_{\mu} \\
  1            & z            &  \cdots  &  z^{m}     \\
 \end{vmatrix}.
 \end{equation}
\noindent Note that the polynomials $\tilde{Q}_{n,m}(z)$ satisfy
\begin{equation}\label{orthpad}
\langle \tilde{Q}_{n,m} F ,p_\nu \rangle_{\mu}=0, \quad \nu=n+1,\ldots, n+m,
\end{equation} and
if we show that $\Delta_{n,m}(F,\mu)\not=0$ (the coefficient of $\tilde{Q}_{n,m}(z)/\prod_{j=1}^m c_{n+j}$), which will be verified at the end of this proof, then ${Q}_{n,m}^{\mu}(z)$ is unique and
 $$Q_{n,m}^{\mu}(z)=\frac{\tilde{Q}_{n,m}(z)}{\Delta_{n,m}(F,\mu)\prod_{j=1}^m c_{n+j}}.$$
Using Cauchy's integral formula and Fubini's theorem, for $\varepsilon>0$ sufficiently small so that $F(z)$ is holomorphic on $D_{1+\varepsilon}$, for $j=1,\ldots,m+1,$ and $\nu=1,\ldots,m$, we have
 \begin{align*}
&c_{n+\nu} \langle z^{j-1} F, p_{n+\nu} \rangle_{\mu} = c_{n+\nu} \int \frac{1}{2 \pi i} \int_{\Gamma_{1+\varepsilon}} \frac{\zeta^{j-1} F(\zeta)}{\zeta-z} d\zeta \overline{p_{n+\nu}(z)} d\mu(z) \notag\\
                 & = \frac{c_{n+\nu} }{2 \pi i}  \int_{\Gamma_{1+\varepsilon}} \zeta^{j-1} F(\zeta) \int \frac{\overline{p_{n+\nu}(z)}}{\zeta-z}  d\mu(z) d\zeta= \frac{c_{n+\nu} }{2 \pi i} \int _{\Gamma_{1+\varepsilon}} \zeta^{j-1} F(\zeta) s_{n+\nu}(\zeta) d\zeta \notag\\
                  &=\frac{c_{n+\nu} }{2 \pi i} \int_{\gamma_{1+\varepsilon}} \Psi^{j-1}(w) f(w)s_{n+\nu}(\Psi(w))\Psi'(w)dw = [ f(w)w^{-\nu} h_{n+\nu}(w)   \Psi^{j-1}(w)  ]_n.
  \end{align*}
 Computing the determinant in (\ref{qnmcircle}) expanding along the last row and applying the previous formula, we obtain
  \begin{equation}\label{lk}
  \tilde{Q}_{n,m}(z )= \sum_{k=0}^m (-1)^{m+k} z^k \det([f K_{n,t} L_{n,r} ] _n)_{t=1,\ldots, m,\, r=1,\ldots,k, k+2, \ldots,m+1},
  \end{equation}
where $$K_{n,t} (w):=w^{-t} h_{n+t}(w), \quad t=1,\ldots, m,$$
  $$L_{n,r}(w):=\Psi^{r-1}(w), \quad r=1,\ldots,m+1.$$
Moreover, all the functions $K_{n,t}(w)$ and $L_{n,r}(w),$ are holomorphic in $U\setminus\{\infty\},$ and
 $$K_t (w):=\lim_{n\rightarrow \infty}K_{n,t}(w)=w^{-t} h(w),\quad t=1,\ldots, m,$$
 $$L_{r}(w):= \Psi^{r-1}(w), \quad r=1,\ldots,m+1,$$
uniformly inside $U\setminus\{\infty\}$.
  By Lemma \ref{Buslaev2} and (\ref{lk}), we have that $\tau_1, \ldots, \tau_m \in U$ and
  \begin{align}
 &\lim_{n \rightarrow \infty} \frac{\tilde{Q}_{n,m}(z)}{\det(f_{n-i-j})_{i,j=0,1,\ldots, m-1}} \notag\\
=&\lim_{n \rightarrow \infty} \sum_{k=0} ^m (-1)^{m+k}  z^k \frac{\det([fK_{n,t} L_{n,r}]_n)_{t=1,\ldots,m,\,\,r=1,\ldots,k,k+2,\ldots,m+1}}{\det(f_{n-i-j})_{i,j=0,1,\ldots, m-1}}  \notag\\
  =& \sum_{k=0} ^m (-1)^{m+k}  z^k \frac{\det(K_r(\tau_t))_{t,r=1,\ldots, m}\det(L_r(\tau_t))_{ t=1,\ldots,m, \,\, r=1,\ldots,k,k+2,\ldots,m+1}}  {W^2(\tau_1,\tau_2,\ldots,\tau_m)} \notag\\
  =& \frac{\det(K_r(\tau_t))_{r,t=1,2,\ldots, m}}{W^2(\tau_1,\tau_2,\ldots,\tau_m)}
  \begin{vmatrix}
  1& \Psi(\tau_1) &  \cdots & \Psi^m(\tau_1)  \\
   \vdots   & \vdots    & \vdots &  \vdots  \\
  1 & \Psi(\tau_m)&  \cdots &  \Psi^m(\tau_m) \\
  1            & z            &  \cdots  &  z^m     \\
 \end{vmatrix} \notag\\=&(-1)^{(m)(m-1)/2}\frac{\prod_{i=1}^m h(\tau_i)}{\prod_{i=1}^m \tau_i^m} \prod_{1 \leq i<j \leq m}\left(\frac{\Psi(\tau_{j})-\Psi(\tau_i)}{\tau_j-\tau_i} \right) z^m+\ldots,\notag
\end{align}
where ${W(\tau_1,\tau_2,\ldots,\tau_m)}=\det(\tau_t^{r-1})_{t,r=1,\ldots,m}$ is the Vandermonde determinant of the numbers $\tau_1,\ldots, \tau_m.$ Since the degree of the polynomial in the last expression is $m,$ the degree of $\tilde{Q}_{n,m}(z)$ is $m$ for all $n$ sufficiently large. This implies that $\Delta_{n,m}(F,\mu)\not=0$ and $Q_{n,m}^{\mu}(z)=\tilde{Q}_{n,m}(z)/(\Delta_{n,m}(F,\mu)\prod_{j=1}^m c_{n+j}).$ Moreover, the zeros of the polynomial in the second last equality are $\lambda_1,\ldots,\lambda_m$, so the zeros of $\tilde{Q}_{n,m}(z)$ (and $Q_{n,m}^{\mu}(z)$) converge to $\lambda_1,\ldots,\lambda_m,$ as $n\rightarrow \infty.$
\end{proof}

\section{Acknowledgement}
The first author would like to thank Brian Simanek for very useful discussions and references.

\noindent Nattapong Bosuwan, Edward B. Saff\\
Center for Constructive Approximation\\
         Department of Mathematics\\
         Vanderbilt University\\
         1326 Stevenson Center\\
         37240 Nashville, TN, USA\\
         \noindent email: nattapong.bosuwan@vanderbilt.edu,
 edward.b.saff@vanderbilt.edu \\

\noindent G. L\'{o}pez Lagomasino\\
Departamento de Matem\'aticas \\
Universidad Carlos III de Madrid \\
c/ Avda. de la Universidad, 30\\
28911, Legan\'es, Spain\\
\noindent email: lago@math.uc3m.es

\end{document}